\documentclass[11pt]{article}
\usepackage[letterpaper,margin=1in]{geometry}

\usepackage{xcolor}
\definecolor{darkblue}{rgb}{0,0,0.5}
\definecolor{darkred}{rgb}{.8,0,0}
\definecolor{darkgreen}{rgb}{0,.5,0}
\usepackage[colorlinks=true,linkcolor=darkblue,citecolor=darkblue,urlcolor=darkblue]{hyperref}

\usepackage{amssymb, url}
\usepackage{amsthm}
\usepackage{amsmath}
\usepackage{verbatim}
\usepackage{graphicx}
\usepackage{caption}
\usepackage{subcaption}
\usepackage{epstopdf}
\usepackage{setspace}
\usepackage{enumerate}
\usepackage{float}
\usepackage{tikz}
	\usetikzlibrary{arrows}
	\usetikzlibrary{calc}
\usepackage{array}
\usepackage{longtable}
\usepackage{underscore}

\newcommand{\ctt}{3\mathord{\circlearrowleft} 3}
\newcommand{\ctf}{3\mathord{\circlearrowleft} 4}
\newcommand{\ctF}{3\mathord{\circlearrowleft} 5}
\newcommand{\cts}{3\mathord{\circlearrowleft} 6}
\newcommand{\cff}{4\mathord{\circlearrowleft} 4}
\newcommand{\dtft}{3\mathord{\xleftrightarrow{1}}3}
\newcommand{\dtFt}{3\mathord{\xleftrightarrow{2}}3}
\newcommand{\dtst}{3\mathord{\xleftrightarrow{3}}3}
\newcommand{\dtff}{3\mathord{\xleftrightarrow{1}}4}
\newcommand{\dtFf}{3\mathord{\xleftrightarrow{2}}4}
\newcommand{\dtsf}{3\mathord{\xleftrightarrow{3}}4}
\newcommand{\dfff}{4\mathord{\xleftrightarrow{1}}4}
\newcommand{\dfFf}{4\mathord{\xleftrightarrow{2}}4}
\newcommand{\dfsf}{4\mathord{\xleftrightarrow{3}}4}
\newcommand{\cfFF}{4\mathord{\circlearrowleft} 55}
\newcommand{\cfFs}{4\mathord{\circlearrowleft} 56}
\newcommand{\cfss}{4\mathord{\circlearrowleft} 66}
\newcommand{\cFFs}{5\mathord{\circlearrowleft} 56}
\newcommand{\cFss}{5\mathord{\circlearrowleft} 66}
\newcommand{\rFFF}{5\mathord{\circlearrowleft} 5\mathord{*}5}

\newcommand{\cfFzF}{4\mathord{\circlearrowleft} 5\mathord{*}5}
\newcommand{\cfFzs}{4\mathord{\circlearrowleft} 5\mathord{*}6}
\newcommand{\cfszs}{4\mathord{\circlearrowleft} 6\mathord{*}6}
\newcommand{\cFfzF}{5\mathord{\circlearrowleft} 4\mathord{*}5}
\newcommand{\cFFFzs}{5\mathord{\circlearrowleft} 55\mathord{*}6}
\newcommand{\cStzF}{7\mathord{\circlearrowleft} 3\mathord{*}5}
\newcommand{\cStzzF}{7\mathord{\circlearrowleft} 3\mathord{*}\mathord{*}5}
\newcommand{\cSfzF}{7\mathord{\circlearrowleft} 4\mathord{*}5}
\newcommand{\cSfzzF}{7\mathord{\circlearrowleft} 4\mathord{*}\mathord{*}5}
\newcommand{\cSFFzF}{7\mathord{\circlearrowleft} 55\mathord{*}5}
\newcommand{\cetzFFzFF}{8\mathord{\circlearrowleft} 3\mathord{*}55\mathord{*}55}

\newcommand{\dfofo}{4\mathord{\xleftrightarrow{\leq 3}}4}
\newcommand{\dthfo}{3\mathord{\xleftrightarrow{\leq 3}}4}
\newcommand{\dthth}{3\mathord{\xleftrightarrow{\leq 3}}3}
\newcommand{\ctsM}{3\mathord{\circlearrowleft} 6^-}

\newcommand{\PC}{Precolorings: \newline}
\newcommand{\RT}{Runtime$^\star$: \newline}
\newcommand{\RTC}{Runtime$^\dagger$: \newline}

\newcommand{\whisk}[2]{%
\draw #1 -- ++(#2:1cm) 		coordinate (A)
	-- ++(#2+30:1cm)		coordinate (B)
	(A) -- ++(#2-30:1cm)		coordinate (C);
	
\draw [fill=white] (A) circle (5.5pt);
\draw [fill=white] (B) circle (5.5pt);
\draw [fill=white] (C) circle (5.5pt);
}
\newcounter{results}

\newtheorem{thm}[results]{Theorem}

\newtheorem{lem}[results]{Lemma}

\newtheorem{conj}[results]{Conjecture}

\makeatletter
\newcommand\xleftrightarrow[2][]{%
  \ext@arrow 9999{\longleftrightarrowfill@}{#1}{#2}}
\newcommand\longleftrightarrowfill@{%
  \arrowfill@\leftarrow\relbar\rightarrow}
\makeatother

\title{The chromatic number of the square of subcubic planar graphs}

\author{
Stephen G. Hartke\thanks{
    Dept.\ of Math.\ and Stat.\ Sciences,
    Univ.\ of Colorado Denver,
    stephen.hartke@ucdenver.edu.  Research supported in part by a Collaboration Grant from the Simons Foundation (\#316262 to Stephen G. Hartke).}
\and
Sogol Jahanbekam\thanks{
    Dept. of Math. and Stat. Sciences,
    Univ. of Colorado Denver,
    Sogol.Jahanbekam@ucdenver.edu}
\and
Brent Thomas\thanks{
    Dept. of Math. and Stat. Sciences,
    Univ. of Colorado Denver,
    Brent.Thomas@ucdenver.edu}
}

\begin{document}
\newcommand{\lbl}[1]{\mathcal{L}(#1)}
\maketitle

\begin{abstract}
  Wegner conjectured in 1977 that the square of every planar graph with maximum degree at most $3$ is $7$-colorable.
  We prove this conjecture using the discharging method and computational techniques to verify reducible configurations.
  \bigskip

  \noindent \emph{Mathematics Subject Classification}: Primary 05C15; Secondary 05C10, 68R10.
  
  \noindent \emph{Keywords}: coloring, square, subcubic, planar graph, discharging, computational proof.
\end{abstract}

\section{Introduction}

Given a simple graph $G$ with vertex set $V(G)$, the \emph{square} of $G$, denoted $G^2$, is the simple graph with vertex set $V(G)$ where vertices $x$ and $y$ are adjacent in $G^2$ if and only if the distance in $G$ between $x$ and $y$ is at most two.
In 1977, Wegner conjectured the following upper bounds on the chromatic numbers of squares of planar graphs.

\begin{conj}[Wegner~\cite{Wegner1977}]\label{conj:Wegner}
  Let $G$ be a planar graph with maximum degree $\Delta$. 
  Then
  \[
  \chi(G^2)\le
  \begin{cases}
    7, & \text{ if $\Delta\le 3$,}\\
    \Delta+5, & \text{ if $4\le \Delta \le 7$,}\\
    \left\lfloor\frac{3\Delta}{2}\right\rfloor+1, & \text{ if $\Delta\ge 8$.}
  \end{cases}
  \]
\end{conj}

In this paper we prove Wegner's conjecture when the maximum degree $\Delta$ is at most $3$.  We use the term \emph{subcubic} for a graph with maximum degree at most $3$.

\begin{thm}\label{main}
  If $G$ is a subcubic planar graph, then $\chi(G^2)\leq 7$.
\end{thm}

Note that the theorem is sharp, as shown by the graph $G$ in Figure~\ref{SharpnessExample}.
In fact, $G^2$ is a complete graph on $7$ vertices.
Since $G$ has vertices of degree $2$, there are an infinite number of connected subcubic planar graphs that contain $G$ and thus have chromatic number $7$ of their squares.

\begin{figure}
  \centering
	\begin{tikzpicture}[scale = .35]
		\input{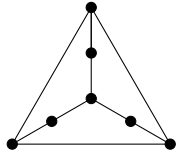}
	\end{tikzpicture}

      \caption{\label{SharpnessExample}A planar graph $G$ with $\chi(G^2)=7$.}
\end{figure} 

In the same paper that Wegner posed Conjecture~\ref{conj:Wegner}, he proved that $\chi(G^2)\leq 8$ for subcubic planar graphs.
In an unpublished work, Borodin (see \cite{JT1995}) proved that $\chi(G^2)\leq 7$ when $G$ is a subcubic planar graph and has no face of size greater than $5$.
Results of Cranston and Kim~\cite{CranstonKim2008} imply that $\chi(G^2)\leq 7$ when $G$ is a subcubic graph with girth at least $7$.
Thomassen announced a proof of Wegner's conjecture for subcubic graphs in 2006, but the proof has not yet appeared.  Our approach uses the discharging method and computation to check reducibility, which is different from Thomassen's approach~\cite{Thomassen}.

Many results related to Theorem~\ref{main} have been proven, including results on the list chromatic number $\chi_\ell$ of the square of planar graphs. 
For a thorough review of the history of coloring squares of planar graphs, we refer the reader to the survey by Borodin~\cite{Borodin2013survey}. 
Here we present the most related results.

A lower bound on the girth forces sparsity in the graph, which is often enough to prove an upper bound on the chromatic and list chromatic numbers. 
Cranston and Kim~\cite{CranstonKim2008} proved that $\chi_{\ell}(G^2)\leq 8$ for any connected subcubic graph $G$ (not necessarily planar) other than the Petersen graph.  
They also proved that for any subcubic planar graph $G$ with girth at least $7$, $\chi_{\ell}(G^2)\leq 7$ holds.
Cranston and Kim~\cite{CranstonKim2008} and independently Havet~\cite{Havet2009} proved that if $G$ is subcubic and planar with girth at least $9$, then $\chi_{\ell}(G^2)\leq 6$. 
Borodin and Ivanova~\cite{BorodinIvanova2012} proved that having girth at least $12$ in any subcubic planar graph $G$ implies $\chi(G^2)\leq 5$. 
They also proved in \cite{BorodinIvanova2012-2} that for all subcubic planar graphs $G$ of girth at least $24$, $\chi(G^2)\leq 4$.  
Note that $4$ is the best possible upper bound for graphs with maximum degree $3$, as any vertex of degree $3$ and its neighbors must have different colors in a proper coloring of the square of the graph. 

Conjecture~\ref{conj:Wegner} has also been studied for planar graphs of higher maximum degree. 
Cranston, Erman, and \v{S}krekovski~\cite{CranstonErmanSkrekovski2014} proved that if $G$ is a planar graph of maximum degree at most $4$, then having girth at least 16, 11, 9, 7, 5, and 3 guarantees $\chi_{\ell}(G^2)$ to be at most 5, 6, 7, 8, 12, and 14, respectively.  
Toward an upper bound for the general case, Molloy and Salavatipour~\cite{MolloySalavatipour2005} showed that $\chi(G^2)\leq \lceil \frac{5}{3}\Delta(G)\rceil+78$. 
Havet, van den Heuvel, McDiarmid, and Reed~\cite{HHMR2007,HHMR2008online} proved that Wegner's Conjecture is asymptotically correct, showing that $\chi(G^2)\leq \frac{3}{2}\Delta(G)(1+o(1))$ for all planar graphs $G$.

The chromatic number of the square of a graph is also studied under the name $2$-distance coloring, which is denoted by $\chi_2$. 
Generalizations of colorings of squares of graphs have also been studied, including distance coloring and $L(p,q)$-labeling of graphs (see for example  \cite{Calamoneri2006,CMR2013,Kramer2008,Yeh2006}).
The natural edge variation of this problem has been also studied under the name strong chromatic index of graphs (see for example \cite{BorodinIvanova2014,Erdos1988,FSGT1990,HHMR2013}). 

Our approach to proving Theorem~\ref{main} uses discharging.
The survey by Cranston and West~\cite{CranstonWest2013} provides a nice overview of the technique.  
Discharging was most famously used in the proofs of the Four Color Theorem by Appel and Haken with Koch~\cite{AppelHaken1977discharging,AppelHaken1977reducibility,AppelHaken1989book}, Robertson, Sanders, Seymour, and Thomas~\cite{RSST1997}, and Steinberger~\cite{Steinberger2010}.
Each of the proofs used computers to verify reducibility of configurations, and the last two proofs also used computers to verify the discharging rules.

Our proof of Theorem~\ref{main} is similar in spirit in that we also verify reducibility by computer, but our discharging rules are simple enough to verify by hand.  
In Sections~\ref{sec:reducible} and \ref{sec:table} we describe the reducible configurations and the verification of their reducibility, and in Section~\ref{sec:computation} we give details of the computation.
In Section~\ref{sec:discharging} we use discharging to prove that our set of reducible configurations is unavoidable in a minimal counterexample, thereby obtaining a contradiction and thus proving Theorem~\ref{main}.  
We conclude in Section~\ref{sec:future} with some conjectures and questions for future work.

In this paper we consider only simple finite undirected graphs.
If $G$ is a planar graph, we assume that it has a fixed planar embedding with no crossing edges.
We denote the vertex set by $V(G)$ and the set of faces by $F(G)$, and we use $\ell(f)$ to denote the length of face $f$.
For convenience, we say that an \emph{$r$-face} of $G$ is a face with length $r$, 
an \emph{$r^+$-face} is a face with length at least $r$, and
an \emph{$r^-$-face} is a face with length at most $r$.  
We call two faces \emph{adjacent} if they share at least one edge in their boundaries. 
For other definitions or notation not given, we refer the reader to the textbook by West~\cite{West2000}.

\section{Minimal counterexamples and reducible configurations}\label{sec:reducible}

We consider a \emph{minimal counterexample} to Theorem~\ref{main}: a subcubic planar graph whose square is not $7$-colorable and that has the fewest number of vertices among all such counterexamples.
We use the term \emph{cubic} for a graph that is regular of degree $3$.

\begin{lem}\label{cubic}
If $G$ is a minimal counterexample to Theorem~\ref{main}, then $G$ is cubic and $3$-connected.
\end{lem}

\begin{proof}
Suppose $G$ contains a vertex $v$ of degree at most $2$.  Let $H$ be the graph $G-v$.
If $v$ has degree $2$ and the neighbors of $v$ are not adjacent, add the edge making them adjacent in $H$. 
Since $H$ is subcubic, planar, and has fewer vertices than $G$, $\chi(H^2)\le 7$.
Note that two different neighbors of $v$ are adjacent in $G^2$, and since these vertices were made adjacent in $H$, they receive different colors in every $7$-coloring of $H^2$.
A $7$-coloring of $H^2$ extends to a $7$-coloring of $G^2$ since $v$ has at most six neighbors in $G^2$.
This contradicts the fact that $G^2$ is not $7$-colorable.  Hence $G$ is cubic.

In a cubic graph, the vertex connectivity is equal to the edge connectivity (see for example Theorem~4.1.11 in \cite{West2000}).  Hence we only need to show that $G$ is $3$-edge-connected.  If $G$ is not connected, then the square of each component of $G$ is $7$-colorable by the minimality of $G$, and together these colorings provide a $7$-coloring of $G^2$.  

If $G$ is connected but not $2$-edge-connected, then $G$ contains a cut edge $e$.  Let $x$ and $y$ be the endpoints of $e$.
Let $H_x$ be the component of $G-e$ containing $x$, and similarly define $H_y$.
Since $H_x$ and $H_y$ are subcubic and planar and have fewer vertices than $G$, then by the minimality of $G$ both $H_x^2$ and $H_y^2$ are $7$-colorable.  Let $\phi_x$ and $\phi_y$ be $7$-colorings of $H_x^2$ and $H_y^2$, respectively.
  
Let $w_1$ and $w_2$ be the neighbors of $x$ in $H_x$, and let $z_1$ and $z_2$ be the neighbors of $y$ in $H_y$.
Note that $x$, $w_1$, and $w_2$ are pairwise adjacent in $H_x^2$, and so receive different colors in $\phi_x$.  Similarly, $y$, $z_1$, and $z_2$ are colored differently in $\phi_y$.
The only additional adjacencies in $G^2$ that are not present in $H_x^2 \cup H_y^2$ are among the vertices $\{x,y,w_1,w_2,z_1,z_2\}$.
We thus permute the names of the colors in $\phi_y$ so that six different colors are assigned to these six vertices in $\phi_x$ and $\phi_y$.
Together $\phi_x$ and $\phi_y$ is a $7$-coloring of $G^2$.
Hence, $G$ is $2$-edge-connected.

If $G$ is not $3$-edge-connected, then $G$ contains an edge cut $\{e_1,e_2\}$. Since $G$ is cubic and $2$-edge-connected, the edges $e_1$ and $e_2$ have distinct endpoints.  
Label the endpoints of $e_1$ as $x_1$ and $y_1$ and the endpoints of $e_2$ as $x_2$ and $y_2$,
so that $x_1$ and $x_2$ belong to the same component $H_x$ of $G-\{e_1,e_2\}$ and $y_1$ and $y_2$ belong to the same component $H_y$ of $G-\{e_1,e_2\}$.

Let $H'_x$ be the graph formed by adding the edge $x_1 x_2$ to $H_x$.  
If $x_1 x_2$ is already present in $H_x$, then $H'_x$ is just $H_x$.  
Similarly, form $H'_y$ by adding the edge $y_1 y_2$ to $H_y$.
The graphs $H'_x$ and $H'_y$ are shown in Figure~\ref{fig:ThreeEdgeConn}.
Each of $H'_x$ and $H'_y$ is subcubic and planar and has fewer vertices than $G$.
Therefore $\chi({H'_x}^2)\leq 7$ and $\chi({H'_y}^2)\leq 7$. 
Let $\phi_x$ and $\phi_y$ be $7$-colorings of ${H'_x}^2$ and ${H'_y}^2$, respectively.  
Note that $\phi_x$ is a $7$-coloring of $H_x^2$ where $x_1$ and $x_2$ receive different colors, 
and that $\phi_y$ is a $7$-coloring of $H_y^2$ where $y_1$ and $y_2$ receive different colors.

\begin{figure}
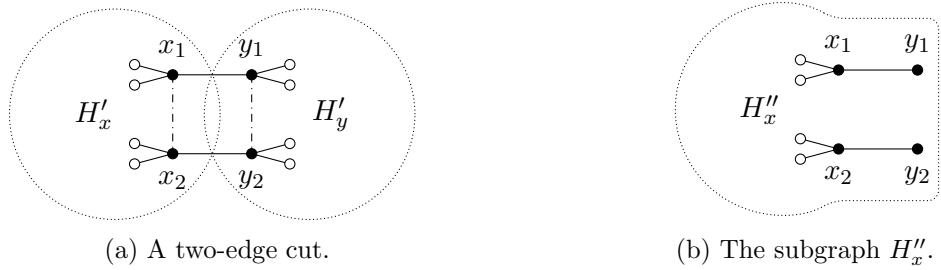

  \centering
    \begin{subfigure}[b]{0.5\textwidth}
    \centering
	\begin{tikzpicture}[scale = .35]
		\input{images/threeedge}
	\end{tikzpicture}

            \caption{A two-edge cut.}
        \label{fig:ThreeEdgeConn}
    \end{subfigure}
    ~ %
      \begin{subfigure}[b]{0.4\textwidth}
      \centering
	\begin{tikzpicture}[scale = .35]
		\input{images/threeedge2}
	\end{tikzpicture}

            \caption{The subgraph $H_x''$.}%
        \label{fig:ThreeEdgeConn2}
    \end{subfigure}
    \caption{The case when $G$ has a two-edge cut.} %
\end{figure} 

Let $H''_x$ be the subgraph of $G$ induced by $V(H_x)\cup \{y_1,y_2\}$, and 
let $H''_y$ be the subgraph of $G$ induced by $V(H_y)\cup \{x_1,x_2\}$.
The graph $H''_x$ is shown in Figure~\ref{fig:ThreeEdgeConn2}.
We extend $\phi_x$ to a $7$-coloring $\phi''_x$ of ${H''_x}^2$ where $x_1$, $x_2$, $y_1$, and $y_2$ all receive different colors.
Note that $y_1$ has at most four colors precluded under $\phi_x$ by $x_2$ and at most three neighbors in ${H''_x}^2$, and so a color may be chosen for $y_1$.
After coloring $y_1$, the vertex $y_2$ has at most five colors precluded by $x_1$, $y_1$, and at most three neighbors in ${H''_x}^2$, and so a color may be chosen for $y_2$. 
Thus $\phi''_x$ is a $7$-coloring of ${H''_x}^2$ where $x_1$, $x_2$, $y_1$, and $y_2$ all receive different colors.
Similarly, we extend $\phi_y$ to a $7$-coloring $\phi''_y$ of ${H''_y}^2$ where $x_1$, $x_2$, $y_1$, and $y_2$ all receive different colors.
Additionally, we permute the name of colors in $\phi''_y$ so that $\phi''_y$ assigns the same colors to $x_1$, $x_2$, $y_1$, and $y_2$ as $\phi''_x$.

Since $\{x_1, x_2, y_1, y_2\}$ is a cut set in $G^2$, there are no additional adjacencies in $G^2$ that are not present in ${H''_x}^2 \cup {H''_y}^2$.
Thus $\phi''_x$ and $\phi''_y$ together is a $7$-coloring of $G^2$, contradicting the fact that $G$ is a counterexample to Theorem~\ref{main}.
Hence $G$ is $3$-edge-connected. 
\end{proof}

We now proceed to find subgraphs that cannot appear in a minimal counterexample.
Following traditional terminology used in discharging proofs, an induced subgraph that cannot appear in a minimal counterexample is called a \emph{reducible configuration}.

Table~\ref{reducible-configurations} contains a list of 31 reducible configurations that we will use to prove Theorem~\ref{main}. 
For each configuration $H$, the black vertices are the vertex set of the induced subgraph $H$ and the white vertices are a subset of $V(G)-V(H)$ that are distance at most 2 from $H$. 
We use two types of notation to describe the configurations in Table~\ref{reducible-configurations}. 

First, to describe a face and some of the adjacent faces we use the symbol $\circlearrowleft$.
The length of the central face appears before $\circlearrowleft$, and a list of lengths of faces appearing in consecutive order counterclockwise around the central face appears afterwards.  
Note that the list might be shorter than the length of the central face, in which case the remaining adjacent faces do not have specified lengths.  
For example, a 4-face with two consecutive adjacent 5-faces is denoted $\cfFF$. 
We also use $*$ to denote adjacent faces whose lengths are unspecified.  
For example, a 4-face with two adjacent 5-faces that are not adjacent to each other is denoted $\cfFzF$. 

Second, we use $\xleftrightarrow{k}$ to denote two faces at distance $k$.
Recall that the distance between two faces is the shortest distance between a vertex of one face and a vertex of the other face.
For example, two $3$-faces at distance $2$ is denoted $\dtFt$. 
For convenience we use $\xleftrightarrow{\leq k}$ to denote two faces at distance at most $k$; this represents $k$ configurations.
Note that in a subcubic graph, faces at distance $0$ share an edge.

\begin{lem}\label{lem:no-reducible-configs}
  None of the 31 configurations in Table~\ref{reducible-configurations} appear in a minimal counterexample to Theorem~\ref{main}.
\end{lem}

\begin{proof}
Let $G$ be a minimal counterexample to Theorem~\ref{main},
and suppose that a configuration $H$ from Table~\ref{reducible-configurations} appears in $G$.
Let $G-H$ denote the subgraph of $G$ induced by $V(G)-V(H)$.
Note that $(G-H)^2$ is $7$-colorable since $G-H$ is a subcubic planar graph with fewer vertices than $G$.
We call a $7$-coloring of $(G-H)^2$ a \emph{precoloring}.
To show that $H$ is reducible, we show that there exists a precoloring that extends to a $7$-coloring of $G^2$.

For most of our configurations, we verify that in fact every precoloring extends to a $7$-coloring of $G^2$.  
However, this is not true for some of our configurations, such as $\dtFf$ and $\cSfzF$.
To find a precoloring that extends, we add edges to $G-H$ to form a graph $D$.
The edges are added to $G-H$ so that $D$ is subcubic and planar.
Since $D$ has fewer vertices than $G$, $D^2$ is $7$-colorable.
Since $(G-H)^2$ is a subgraph of $D^2$, a $7$-coloring of $D^2$ is also a $7$-coloring of $(G-H)^2$, but is a precoloring with additional properties because of the extra edges in $D$.
We verify that every $7$-coloring of $D$ is a precoloring that extends to a $7$-coloring of $G^2$.

Edges that are added to $G-H$ to form $D$ are shown as dashed edges in the figures of Table~\ref{reducible-configurations}.
Note that edges are added only between vertices of $G-H$ (drawn as white vertices) that are neighbors of vertices in $H$ (drawn as black vertices).
The fact that the graph $D$ formed by adding dashed edges is subcubic and planar can be immediately verified by inspection of the figures.
However, we need to verify that the endpoints of an added dashed edge are in fact distinct vertices.

Let the vertices on the outer face of a configuration $H$ be labeled clockwise $v_1, \ldots, v_s$. 
If $v_i$ and $v_k$ are vertices of $H$ such that $\deg_H(v_i)=\deg_H(v_k)=2$, $2\leq |i-k|\leq 3$ and there is a single vertex $v_j$ such that $\deg_H(v_j)=2$ and $i < j < k$, then we claim that the neighbors of $v_i$ and $v_k$ in $G-H$ are distinct. 
If not, then let $u$ be the common neighbor of $v_i$ and $v_k$ in $G-H$. 
 
Since $G$ is cubic and 3-connected, there must exist a path between $u$ and $v_j$ which is inside the closed curve containing $u,v_i,v_j,v_k$, since otherwise $v_j$ is a cut vertex. 
If this path is of length greater than $1$, i.e. $uv_j\notin E(G)$, then $\{u, v_j\}$ is a cut set of size $2$, as shown in Figure~\ref{fig:AddEdge1}.
Otherwise $uv_j \in E(G)$, as shown in Figure~\ref{fig:AddEdge2}.
In this case, when $|i-k|=2$ (such as in $\dtFf$ and $\cfFzF$), $G$ contains $\ctt$, and when $|i-k|=3$ (as in $\cSfzF$ and $\cSfzzF$), $G$ contains $\ctf$, both of which are reducible configurations without adding any edges to $G-H$. 
Thus a dashed edge between the neighbors in $G-H$ of $v_i$ and $v_k$ can be added to form the graph $D$ and to find precolorings with desired properties.

\begin{figure}
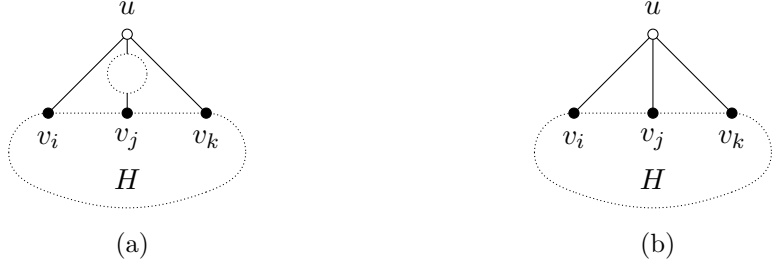

  \centering
    \begin{subfigure}[b]{0.4\textwidth}
    \centering
	\begin{tikzpicture}[scale = .35]
		\input{images/addedge1}
	\end{tikzpicture}

            \caption{}
            \label{fig:AddEdge1}
    \end{subfigure}
    ~ %
      \begin{subfigure}[b]{0.4\textwidth}
      \centering
	\begin{tikzpicture}[scale = .35]
		\input{images/addedge2}
	\end{tikzpicture}

            \caption{}
            \label{fig:AddEdge2}
    \end{subfigure}
    \caption{The cases that occur when the white neighbors of $v_i$ and $v_k$ are the same vertex.}
\end{figure} 

To extend a precoloring, we must $7$-color the vertices of $H$ in a way compatible with the precoloring on $(G-H)^2$.  Thus, we only need to consider the precoloring of vertices of $G-H$ that are within distance $2$ of a vertex of $H$.
In Table~\ref{reducible-configurations}, these vertices of $G-H$ are shown in white.

In general, each vertex of $H$ that has degree $2$ in $H$ is adjacent to one white vertex, which in turn is adjacent to two other white vertices.
However, it is possible that the white vertices as drawn may not all be distinct or there may be adjacencies between white vertices that are not drawn.
Figure~\ref{RestrictedPrecolorings} shows an example of such a situation.
Each precoloring of the white vertices restricts the colors that may appear on the black vertices of the configuration.
Additional relationships among the white vertices limits the precolorings that can appear on the white vertices, 
so the restrictions imposed by precoloring the white vertices are a subset of the restrictions imposed by precolorings of the white vertices when there are no additional relationships among the white vertices (except possibly the added dashed edges discussed above).
Thus if we show that every possible precoloring when assuming no vertex identifications of white vertices or additional adjacencies (beyond the dashed edges) can be extended to a coloring that includes $H$, then we have also shown that every precoloring with vertex identifications or additional adjacencies also extends to a coloring of $G^2$.

\begin{figure}
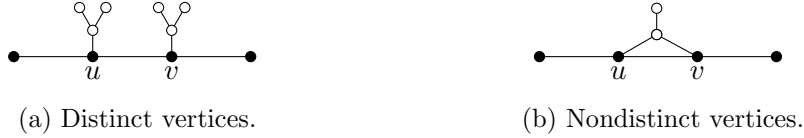

  \centering
    \begin{subfigure}[b]{0.4\textwidth}
    \centering
	\begin{tikzpicture}[scale = .35]
		\input{images/RestrictedPrecolorings}
	\end{tikzpicture}

        \caption{Distinct vertices.}
        \label{fig:DistinctVertices}
    \end{subfigure}
    ~ %
      \begin{subfigure}[b]{0.4\textwidth}
      \centering
	\begin{tikzpicture}[scale = .35]
		\input{images/RestrictedPrecoloringsNonDistinct}
	\end{tikzpicture}

        \caption{Nondistinct vertices.}
        \label{fig:NonDistinctVertices}
    \end{subfigure}
    \caption{Arrangements of the vertices of $G-H$.}
    \label{RestrictedPrecolorings}
\end{figure} 

The verification of extending precolorings is done using a program that verifies every proper $k$-coloring of specified vertices in an input graph can be extended to a proper $k$-coloring of the entire graph.
As the program acts on a single input graph, we need to ensure it correctly generates all precolorings (i.e., $7$-colorings of $G-H$ or $D$).
Let $H$ be a configuration in Table~\ref{reducible-configurations} with no dashed edges, and let $F$ be the graph (including the white vertices) shown for $H$ in Table~\ref{reducible-configurations}.
We input $F^2$ into the program, and ask if every $7$-coloring of the white vertices can be extended to a $7$-coloring of the black vertices.
Let $W$ denote the set of white vertices in $F$, and let $F[W]$ denote the subgraph of $F$ induced by $W$.
For each graph $F$ in the table, note that any two white vertices at distance $2$ are connected by a path of length $2$ that has a white internal vertex.
Then $(F[W])^2$ is the subgraph of $F^2$ induced by $W$, so every precoloring of the white vertices is a proper $7$-coloring of the white vertices in $F^2$, and hence is generated and checked by the program.

Suppose now that $H$ is a configuration in Table~\ref{reducible-configurations} with dashed edges, and again let $F$ be the graph (including the white vertices but not the dashed edges) shown for $H$ in Table~\ref{reducible-configurations}.
Let $F'$ denote $F$ with the dashed edges added.
As above, let $W$ denote the set of white vertices in $F$, and let $F'[W]$ denote the subgraph of $F'$ induced by $W$.
Let $J$ be the graph formed by adding the dashed edges to $F^2$ as well as edges between white vertices connected in $F'$ by a path of length $2$ consisting of one dashed edge and one edge of $F$.
Note that $J$ is a subgraph of $(F')^2$ but does not have all of the edges of $(F')^2$, as a white vertex and a black vertex at distance $2$ in $F'$ are not made adjacent in $J$.

We input the graph $J$ into the program to verify that every proper $7$-coloring of the white vertices extends to a $7$-coloring of the black vertices.
Since we only add edges between white vertices to $F^2$ to form $J$, $(F'[W])^2$ is the subgraph of $J$ induced by $W$, so again we have that every precoloring from $D$ of the white vertices is a proper $7$-coloring of the white vertices in $J$, and hence is generated and checked by the program.

\medskip

The details of the program and the computation verifying the reducibility of the configurations are discussed in Section~\ref{sec:computation}.
As a result of these computations, each configuration in Table~\ref{reducible-configurations} is reducible.
\end{proof}

To verify the reducibility of the configurations, we could have written code that generated every precoloring of $(F[W])^2$ or $(F'[W])^2$, and then in a separate program checked that the precoloring extended to $G^2$.
Instead we wrote a program that takes a single graph as input and that verifies precoloring extension for that graph.
The necessary program was simpler to write, easier to check for correctness, and easier to optimize and parallelize.
We also hope that the program may be of use to other researchers investigating precoloring extensions.

\section{Computation}\label{sec:computation}

\DeclareRobustCommand{\cpluspluslogo}{\hbox{C\hspace{-0.5ex}
                       \protect\raisebox{0.5ex}
                       {\protect\scalebox{0.67}{++}}}}

We wrote an optimized \cpluspluslogo{} program \texttt{precolor_extend}%
\footnote{The code is available at \url{http://www.math.ucdenver.edu/~hartkes/math/data/data.php}.}
to verify precoloring extensions.  
The input is the number $k$ of colors, the number $t$ of vertices to precolor, and a graph $G$ on the vertex set $\{0,1,\dots,n-1\}$.  
The program uses a greedy backtracking search to generate all proper precolorings of the vertices $\{0,1,\dots,t-1\}$ with the $k$ colors $\{0,1,\dots,k-1\}$, and then uses the same greedy backtracking method to see if the precoloring can be extended to a proper coloring of the entire graph with those $k$ colors.
The greedy coloring removes some of the symmetry present in colorings by coloring vertex $0$ with color $0$ and then coloring each new vertex with a previously used color or a new color (chosen to be the smallest unused color), instead of coloring the new vertex with all possible $k$ colors.

The program is implemented to be simple and very fast.  
Bit masks are used to test if any neighbors of a given vertex are colored with a specified color in constant time.
Symmetry of the graph in general is not exploited, except in the case that two consecutive vertices (in the graph ordering) have the same neighborhoods.
This occurs in our configurations for the two leaves on each added stem.
In this case, we require the color of the second vertex to be greater than the color on the first vertex.

For our graphs, we made another simplifying assumption that significantly reduces the number of precolorings that need to be checked.  
In each graph, white leaves restrict the color of only one black vertex.
If that black vertex is also distance $2$ from a white non-leaf $v$, then coloring the white leaves the same as $v$ is redundant.  
Thus we require that a white leaf is colored differently than any other white non-leaf that constrains the same black vertex.
We enforce this constraint by adding edges between the white leaf and these other white non-leaves.  
These edges are shown as dashed-dotted edges in Figure~\ref{fig:PrecoloringAssuption}.
Note that each pair of white leaves has at most three forbidden colors, coming from the common white neighbor and at most two ``$v$'' vertices, and so the two white leaves can always be colored to avoid those three colors.
These dashed-dotted edges are added after forming the square of the graph, and hence occur in the preparation of the input graph given to \texttt{precolor_extend}.
However, if the configuration required dashed edges to restrict the precolorings as described in the previous section, we did not add the dashed-dotted edges, as the interaction between the two types of added edges is hard to determine.

\begin{figure}
  \centering
	\begin{tikzpicture}[scale = .35]
		\input{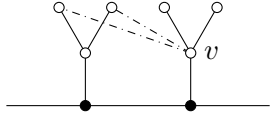}
	\end{tikzpicture}

      \caption{\label{fig:PrecoloringAssuption}Reducing the number of precolorings.}
\end{figure} 

As the running time can be long for large graphs, we enabled a ``root coloring'' to be given that fixes the colors of some of the precolored vertices.  This allowed the set of precolorings to be partitioned and each part checked in parallel.

We used Sage Mathematical Software~\cite{sage} to prepare the input files fed into the \texttt{precolor_extend} program.  
The computations verifying the configurations listed in Table~\ref{reducible-configurations} were performed on two machines.  
The smaller computations were done on the xvib server, a machine with four AMD Opteron~8350 2.0GHz processors, for a total of 16 CPU cores.  
The larger computations were performed on the colibri cluster, which has 48 Intel Xeon~E5-2670 2.6GHz processors, for a total of 384 CPU cores.  
Both machines run CentOS~6.7, and \texttt{precolor_extend} was compiled with GCC~4.4.7.  
Table~\ref{reducible-configurations} contains the running time for each configuration and which machine performed the computation.

\section{Discharging}\label{sec:discharging}

In this section, we prove using discharging that any minimal counterexample to Theorem~\ref{main} contains one of the reducible configurations from Table~\ref{reducible-configurations}.
We use face charging, and the following lemma gives the total initial charge.

\begin{lem}\label{lem:facesum}
If $G$ is a cubic planar graph with a fixed planar embedding, then
\[
\sum_{f\in F(G)} \left(\ell(f)-6\right) = -12.
\]
\end{lem}

\begin{proof}
Let $G$ have $n$ vertices, $m$ edges, and $p$ faces, and let $F(G)$ denote the set of faces of $G$.  
The degree-sum formula applied to $G$ gives $3n=\sum_{v\in V(G)} \deg_G(v)=2m$, and applied to the dual of $G$ gives $\sum_{f\in F(G)} \ell(f)=2m$.  
Thus, $n-m=-\frac{1}{6} \sum_{f\in F(G)} \ell(f)$. 
Substituting into Euler's Formula $n-m+p=2$, we have that
\[
\left(-\frac{1}{6} \sum_{f\in F(G)} \ell(f) \right) +p =2.
\]
Thus,
\[
\sum_{f\in F(G)} \left(\ell(f) -6 \right) = -12.\qedhere
\]
\end{proof}

\begin{proof}[Proof of Theorem~\ref{main}]
Let $G$ be a minimal counterexample to the theorem.
By Lemma~\ref{lem:no-reducible-configs}, $G$ contains none of the $31$ reducible configurations listed in Table~\ref{reducible-configurations}.

Assign each face $f$ of $G$ an initial charge of $\ell(f)-6$. 
Since $G$ is cubic by Lemma~\ref{cubic}, Lemma~\ref{lem:facesum} implies that the total charge is $-12$.

Note that by Lemma~\ref{cubic}, $G$ is $3$-edge-connected, which implies that every face is a cycle and every edge separates two different faces.
Additionally, each $r$-face is adjacent to $r$ different faces.

We apply the following discharging rules. 

\begin{enumerate}[(R1)]
\item Each $3$-face receives charge $1$ from each adjacent $7^+$-face.
\item Each $4$-face receives charge $\frac{2}{3}$ from each adjacent $7^+$-face.
\item Each $5$-face receives charge $\frac{1}{3}$ from each adjacent $7^+$-face.
\end{enumerate}

We prove that the final charge on each face after discharging is nonnegative by considering faces by length.

\medskip\noindent{\emph{$3$-faces.}}
Each $3$-face has initial charge $-3$. 
Since $\ctsM$ are reducible configurations, each $3$-face has three adjacent $7^+$ faces and so receives $3$ charge.
Since a $3$-face gives no charge, each $3$ face has final charge $0$. 

\medskip\noindent{\emph{$4$-faces.}}
Each $4$-face has initial charge $-2$. 
Since the configurations $\ctf$, $\cff$, $\cfFF$, $\cfFs$, $\cfss$, $\cfFzF$, $\cfFzs$, and $\cfszs$ are reducible, every $4$-face has at least three adjacent $7^+$-faces. 
As a result, each $4$-face receives at least $3(\frac{2}{3})$ charge and gives no charge. 
Therefore $4$-faces have final charge at least $0$.

\medskip\noindent{\emph{$5$-faces.}}
Let $f$ be a $5$-face.  The initial charge of $f$ is $-1$.
Reducibility of $\ctF$ implies that $f$ has no adjacent $3$-face. If $f$ has an adjacent $4$-face, then by reducibility of $\dfofo$, $f$ has no other adjacent $4$-face. By the reducibility of $\cfFF$ and $\cfFs$ the faces adjacent to $f$ and the $4$-face must be $7^+$-faces. The faces adjacent to $f$ at distance 1 from the $4$-face must be $6^+$-faces by the reducibility of $\cFfzF$, moreover one of the faces must be a $7^+$-face by the reducibility of $\cFss$.

If $f$ has no adjacent $4$-face, then since the configurations $\cFFs$, $\cFss$, $\rFFF$, and $\cFFFzs$ are reducible, $f$ has at least three adjacent $7^+$-faces. 
In both cases, the $5$-face $f$ gives no charge and receives total charge at least of $3(\frac{1}{3})$. 
Therefore $f$ has final charge at least $0$.

\medskip\noindent{\emph{$6$-faces.}}
Each $6$-face has initial charge $0$.
According to our discharging rules, $6$-faces do not give or receive charge, hence their final charge remains $0$.

\medskip\noindent{\emph{$7$-faces.}}
Let $f$ be a $7$-face.  The initial charge of $f$ is $1$.
First suppose that $f$ has an adjacent $3$-face. Reducibility of $\dthth$ and $\dthfo$ imply that every other face adjacent to $f$ has length at least $5$. However, the reducibility of  $\ctF$, $\cStzF$, and $\cStzzF$ imply that $f$ has no adjacent $5$-face. Therefore $f$ loses total charge of $1$ in this case, resulting in a final charge of $0$.

If $f$ has no adjacent $3$-face but has an adjacent $4$-face, then by the reducibility of $\dfofo$, $f$ has no other adjacent $4^-$-face.
By reducibility of $\cfFzF$, $\cSfzF$, and $\cSfzzF$, the face $f$ has at most one adjacent $5$-face.
Hence $f$ loses at most $\frac{2}{3}$ charge to an adjacent $4$-face and $\frac{1}{3}$ charge to an adjacent $5$-face, resulting in a final charge at least $0$. 

If $f$ has no adjacent $3$-face and no adjacent $4$-face, then by the reducibility of $\rFFF$ and $\cSFFzF$, the face $f$ has at most three adjacent $5$-faces.
Therefore $f$ loses at most $3(\frac{1}{3})$ charge. 
As a result, $f$ has a nonnegative final charge. 

\medskip\noindent{\emph{$8$-faces.}}
Each $8$-face has initial charge $2$. 
Let $f$ be an $8$-face having an adjacent $3$-face. 
The reducibility of $\ctsM$, $\dthth$, and $\dthfo$ imply that $f$ has no other adjacent $3$-face and no adjacent $4$-face. 
Moreover, since $\rFFF$ and $\cetzFFzFF$ are reducible, $f$ has at most three adjacent $5$-faces in this case. 
Therefore $f$ loses at most $1+3(\frac{1}{3})$ charge, resulting in a nonnegative final charge.

Let $f$ be an $8$-face having an adjacent $4$-face and no adjacent $3$-face. 
Since $\cff$ and $\dfofo$ are reducible, $f$ has no other adjacent $4$-face.  
Since $\rFFF$ and $\cfFzF$ are reducible, $f$ has at most four adjacent $5$-faces.
Therefore, $f$ loses at most $\frac{2}{3}+4(\frac{1}{3})$ charge, resulting in a nonnegative final charge.

Let $f$ be an $8$-face having no adjacent $4^-$-face. The reducibility of $\rFFF$ implies that $f$ has at most five adjacent $5$-faces. Therefore $f$ loses at most $5(\frac{1}{3})$ charge. As a result, $f$ has final charge at least $\frac{1}{3}$.

\begin{figure}
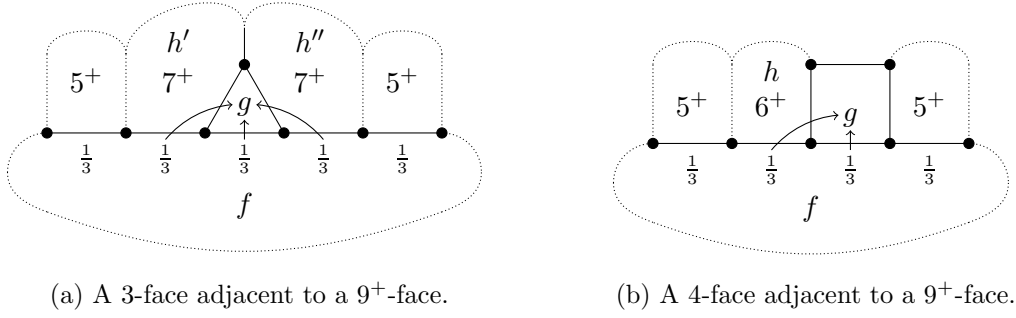

    \centering
    \begin{subfigure}[b]{0.45\textwidth}
       \centering
	\begin{tikzpicture}[scale = .35]
		\input{images/discharge3}
	\end{tikzpicture}

            \caption{A $3$-face adjacent to a $9^+$-face.}
        \label{fig:3face}
    \end{subfigure}
    \quad %
    \begin{subfigure}[b]{0.4\textwidth}
        \centering
	\begin{tikzpicture}[scale = .35]
		\input{images/discharge4}
	\end{tikzpicture}

            \caption{A $4$-face adjacent to a $9^+$-face.}
        \label{fig:4face}
    \end{subfigure}
       \caption{Discharging from a $9^+$-face.}\label{fig:discharging}
\end{figure}

\medskip\noindent{\emph{$9^+$-faces.}}
Let $f$ be a $9^+$-face with length $r$.  The initial charge of $f$ is $r-6$.
To achieve the transfer of charge to neighboring $5^-$-faces described in rules R1--R3, we will first do a preliminary step that transfers charge from $f$ to its bounding edges.
The charge on the edges will then be transferred to adjacent faces, completing the discharging from $f$.

We allocate $\frac{1}{3}$ charge from $f$ to each edge bounding $f$.  
Since $r\ge 9$, then $r-6\ge \frac{1}{3} r$, so the initial charge of $f$ is large enough to perform this allocation.

If $g$ is a $5$-face adjacent to $f$, then we transfer the $\frac{1}{3}$ charge from the edge separating $f$ and $g$ to $g$.  

Let $g$ be a $3$-face adjacent to $f$.
Since the configurations $\ctsM$ are reducible, the two faces $h'$ and $h''$ adjacent to $f$ and $g$ have length at least $7$.
Figure~\ref{fig:3face} depicts this situation.
We transfer the total charge of $1$ on the three edges separating $f$ from $g$, $h'$, and $h''$ to $g$.
Because the configurations $\dthth$ and $\dthfo$ are reducible, the other face adjacent to $f$ and $h'$ has length at least $5$.  
Similarly, the other face adjacent to $f$ and $h''$ also has length at least $5$.
Neither of these two faces will pull charge from the edges separating $f$ from $h'$ and $h''$.

Let $g$ be a $4$-face adjacent to $f$. 
Since configurations $\cff$ and $\cfFzF$ are reducible, at least one of the two faces adjacent to $f$ and $g$ has length at least $6$; call this face $h$ as shown in Figure~\ref{fig:4face}.
We transfer the $\frac{2}{3}$ total charge on the two edges separating $f$ from $g$ and $h$ to $g$.
Because the configurations $\dthfo$ and $\dfofo$ are reducible, the other face adjacent to $f$ and $g$ has length at least $5$.
Similarly, the other face adjacent to $f$ and $h$ also has length at least $5$.
Neither of these two faces will pull charge from the edges separating $f$ from $g$ and $h$.

In all cases, $f$ transfers the charge specified by the discharging rules R1--R3 to adjacent faces, and the final charge of $f$ is nonnegative.

\bigskip

Thus we have shown that after discharging, each face has a nonnegative final charge.  
Since no charge is lost or gained during discharging, this gives us a contradiction of the initial total charge being negative.  
Hence no minimal counterexample exists, and so the square of every subcubic planar graph is $7$-colorable.
\end{proof}

\section{Future work}\label{sec:future}

All of the sharpness examples described after Theorem~\ref{main} have $5$-faces and vertex cutsets of size at most $2$.  These examples motivate the following conjectures refining Theorem~\ref{main}.

\begin{conj}
  If $G$ is a subcubic planar graph drawn without any faces of length 5, then $\chi(G^2)\leq 6$.
\end{conj}

\begin{conj}\label{conj:3-conn}
  If $G$ is a 3-connected cubic planar graph, then $\chi(G^2)\leq 6$.
\end{conj}

Conjecture~\ref{conj:3-conn} was also made by Jensen and Toft~\cite{JT1995}.

The \texttt{precolor_extend} program enabled us to experiment with many configurations to explore reducibility.  
Some reducible configurations that we found (such as $7\mathord{\circlearrowleft}45$) were not necessary for our discharging proof.
Other configurations (such as $4\mathord{\circlearrowleft}5$ and $5\mathord{\circlearrowleft}55$) were not reducible, despite our best efforts in restricting the precolorings.  
Surprisingly, identification of white vertices to restrict the precolorings was unhelpful in proving reducibility.
It would be interesting to determine the minimum number of reducible configurations needed to give a discharging proof of Theorem~\ref{main}.

\section*{Acknowledgements}
The colibri cluster is hosted by the Center for Computational Mathematics at the University of Colorado Denver and was funded by National Science Foundation grant 0958354.
The authors thank 
Jennifer Diemunsch for helpful discussions
and Carsten Thomassen for encouraging them to pursue the approach described in this paper.

{
\small
\frenchspacing
\setlength{\itemsep}{0em}
\setlength{\parskip}{0.25em}

\bibliography{references}
\bibliographystyle{plain}
}

\section{List of reducible configurations}\label{sec:table}
\begin{longtable}{ | m{5cm} | c | }
    \caption{Reducible Configurations}
    \label{reducible-configurations}
    \endfirsthead
    \endhead
    \hline
    Data & Reducible Configuration \\ \hline
      \begin{itemize}
        \item  $\ctt$%
        \item \PC 15 %
        \item \RT less than 1s%
      \end{itemize}
    & 
    \begin{minipage}{.5\textwidth}
      \begin{center}\vspace*{5pt}\vspace*{5pt}
	\begin{tikzpicture}[scale = .35]
		\input{images/c33.tex}
	\end{tikzpicture}

      \vspace*{5pt}\end{center}
      
      \textbf{Note:} This configuration is also reducible by Lemma~\ref{cubic}.
    \end{minipage}
    \\ \hline

     \begin{itemize}
        \item $\ctf$ %
        \item \PC 333%
        \item \RT less than 1s%
      \end{itemize}
    & 
    \begin{minipage}{.5\textwidth}
      \begin{center}\vspace*{5pt}
	\begin{tikzpicture}[scale = .35]
		\input{images/c34.tex}
	\end{tikzpicture}

      \vspace*{5pt}\end{center}
    \end{minipage}
    \\ \hline

     \begin{itemize}
        \item $\ctF$ %
        \item \PC 13,076%
        \item \RT less than 1s%
      \end{itemize}
    & 
    \begin{minipage}{.5\textwidth}
      \begin{center}\vspace*{5pt}
	\begin{tikzpicture}[scale = .35]
		\input{images/c35.tex}
	\end{tikzpicture}

      \vspace*{5pt}\end{center}
    \end{minipage}
    \\ \hline

     \begin{itemize}
        \item $\cts$ %
        \item \PC 608,261%
        \item \RT less than 1s%
      \end{itemize}
    & 
    \begin{minipage}{.5\textwidth}
      \begin{center}\vspace*{5pt}
	\begin{tikzpicture}[scale = .35]
		\input{images/c36.tex}
	\end{tikzpicture}

      \vspace*{5pt}\end{center}
    \end{minipage}
    \\ \hline

     \begin{itemize}
        \item $\dtft$ %
        \item \PC 7,458%
        \item \RT less than 1s%
      \end{itemize}
    & 
    \begin{minipage}{.5\textwidth}
      \begin{center}\vspace*{5pt}
	\begin{tikzpicture}[scale = .35]
		\input{images/d343.tex}
	\end{tikzpicture}

      \vspace*{5pt}\end{center}
    \end{minipage}
    \\ \hline

     \begin{itemize}
        \item $\dtFt $%
        \item \PC 662,720%
        \item \RT 4s%
      \end{itemize}
    & 
    \begin{minipage}{.5\textwidth}
      \begin{center}\vspace*{5pt}
	\begin{tikzpicture}[scale = .35]
		\input{images/d353.tex}
	\end{tikzpicture}

      \vspace*{5pt}\end{center}
    \end{minipage}
    \\ \hline

     \begin{itemize}
        \item $\dtst$ %
        \item \PC 33,059,884%
        \item \RT 11s%
      \end{itemize}
    & 
    \begin{minipage}{.5\textwidth}
      \begin{center}\vspace*{5pt}
	\begin{tikzpicture}[scale = .35]
		\input{images/d363.tex}
	\end{tikzpicture}

      \vspace*{5pt}\end{center}
    \end{minipage}
    \\ \hline

     \begin{itemize}
        \item $\dtff$ %
        \item \PC 357,823%
        \item \RT less than 1s%
      \end{itemize}
    & 
    \begin{minipage}{.5\textwidth}
      \begin{center}\vspace*{5pt}
	\begin{tikzpicture}[scale = .35]
		\input{images/d344.tex}
	\end{tikzpicture}

      \vspace*{5pt}\end{center}
    \end{minipage}
    \\ \hline

     \begin{itemize}
        \item $\dtFf$ %
        \item \PC 138,328,236%
        \item \RT 34s%
      \end{itemize}
    & 
    \begin{minipage}{.5\textwidth}
      \begin{center}\vspace*{5pt}
	\begin{tikzpicture}[scale = .35]
		\input{images/d354.tex}
	\end{tikzpicture}

      \vspace*{5pt}\end{center}
    \end{minipage}
    \\ \hline

     \begin{itemize}
        \item $\dtsf$ %
        \item \PC 13,344,796,170%
        \item \RT 1h 10m 58s%
      \end{itemize}
    & 
    \begin{minipage}{.5\textwidth}
      \begin{center}\vspace*{5pt}
	\begin{tikzpicture}[scale = .35]
		\input{images/d364.tex}
	\end{tikzpicture}

      \vspace*{5pt}\end{center}
    \end{minipage}
    \\ \hline

     \begin{itemize}
        \item $\cff$
        \item \PC 13,488 %
        \item \RT less than 1s%
      \end{itemize}
    & 
    \begin{minipage}{.5\textwidth}
      \begin{center}\vspace*{5pt}
	\begin{tikzpicture}[scale = .35]
		\input{images/c44.tex}
	\end{tikzpicture}

      \vspace*{5pt}\end{center}
    \end{minipage}
    \\ \hline

     \begin{itemize}
        \item $\dfff$%
        \item \PC 134,815,734%
        \item \RT 51s%
      \end{itemize}
    & 
    \begin{minipage}{.5\textwidth}
      \begin{center}\vspace*{5pt}
	\begin{tikzpicture}[scale = .35]
		\input{images/d444.tex}
	\end{tikzpicture}

      \vspace*{5pt}\end{center}
    \end{minipage}
    \\ \hline

     \begin{itemize}
        \item $\dfFf$ %
        \item \PC 5,064,449,220%
        \item \RT 25m 41s%
      \end{itemize}
    & 
    \begin{minipage}{.5\textwidth}
      \begin{center}\vspace*{5pt}
	\begin{tikzpicture}[scale = .35]
		\input{images/d454.tex}
	\end{tikzpicture}

      \vspace*{5pt}\end{center}
    \end{minipage}
    \\ \hline

     \begin{itemize}
        \item $\dfsf$ %
        \item \PC 500,635,773,360%
        \item \RTC  12h 46m 49s%
      \end{itemize}
    & 
    \begin{minipage}{.5\textwidth}
      \begin{center}\vspace*{5pt}
	\begin{tikzpicture}[scale = .35]
		\input{images/d464.tex}
	\end{tikzpicture}

      \vspace*{5pt}\end{center}
    \end{minipage}
    \\ \hline

     \begin{itemize}
        \item $\cfFF$ %
        \item \PC 1,105,527%
        \item \RT 2s%
      \end{itemize}
    & 
    \begin{minipage}{.5\textwidth}
      \begin{center}\vspace*{5pt}
	\begin{tikzpicture}[scale = .35]
		\input{images/c455.tex}
	\end{tikzpicture}

      \vspace*{5pt}\end{center}
    \end{minipage}
    \\ \hline

     \begin{itemize}
        \item $\cfFs$ %
        \item \PC 54,788,105%
        \item \RT 24s%
      \end{itemize}
    & 
    \begin{minipage}{.5\textwidth}
      \begin{center}\vspace*{5pt}
	\begin{tikzpicture}[scale = .35]
		\input{images/c456.tex}
	\end{tikzpicture}

      \vspace*{5pt}\end{center}
    \end{minipage}
    \\ \hline

     \begin{itemize}
        \item $\cfss$ %
        \item \PC 2,802,164,397%
        \item \RT 20m 40s%
      \end{itemize}
    & 
    \begin{minipage}{.5\textwidth}
      \begin{center}\vspace*{5pt}
	\begin{tikzpicture}[scale = .35]
		\input{images/c466.tex}
	\end{tikzpicture}

      \vspace*{5pt}\end{center}
    \end{minipage}
    \\ \hline

     \begin{itemize}
        \item $\cfFzF$ %
        \item \PC 51,915,036%
        \item \RT 24s%
      \end{itemize}
    & 
    \begin{minipage}{.5\textwidth}
      \begin{center}\vspace*{5pt}
	\begin{tikzpicture}[scale = .35]
		\input{images/a4505.tex}
	\end{tikzpicture}

      \vspace*{5pt}\end{center}
    \end{minipage}
    \\ \hline

     \begin{itemize}
        \item $\cfFzs$ %
        \item \PC 1555301069%
        \item \RT 11m 15s%
      \end{itemize}
    & 
    \begin{minipage}{.5\textwidth}
      \begin{center}\vspace*{5pt}
	\begin{tikzpicture}[scale = .35]
		\input{images/a4506.tex}
	\end{tikzpicture}

      \vspace*{5pt}\end{center}
    \end{minipage}
    \\ \hline

     \begin{itemize}
        \item $\cfszs$ %
        \item \PC 1,321,455,661,506%
        \item \RT 7d 3h 51m 16s%
      \end{itemize}
    & 
    \begin{minipage}{.5\textwidth}
      \begin{center}\vspace*{5pt}
	\begin{tikzpicture}[scale = .35]
		\input{images/a4606.tex}
	\end{tikzpicture}

      \vspace*{5pt}\end{center}
    \end{minipage}
    \\ \hline

     \begin{itemize}
        \item $\cFFs$
        \item \PC 2,846,374,664%
        \item \RT 37m 30s%
      \end{itemize}
    & 
    \begin{minipage}{.5\textwidth}
      \begin{center}\vspace*{5pt}
	\begin{tikzpicture}[scale = .35]
		\input{images/c556.tex}
	\end{tikzpicture}

      \vspace*{5pt}\end{center}
    \end{minipage}
    \\ \hline

     \begin{itemize}
        \item $\cFss$ %
        \item \PC 146,973,280,732%
        \item \RT 23h 50m 57s%
      \end{itemize}
    & 
    \begin{minipage}{.5\textwidth}
      \begin{center}\vspace*{5pt}
	\begin{tikzpicture}[scale = .35]
		\input{images/c566.tex}
	\end{tikzpicture}

      \vspace*{5pt}\end{center}
    \end{minipage}
    \\ \hline

     \begin{itemize}
        \item $\cFfzF$ %
        \item \PC 54,788,105%
        \item \RT 21s%
      \end{itemize}
    & 
    \begin{minipage}{.5\textwidth}
      \begin{center}\vspace*{5pt}
	\begin{tikzpicture}[scale = .35]
		\input{images/a5405.tex}
	\end{tikzpicture}

      \vspace*{5pt}\end{center}
    \end{minipage}
    \\ \hline
    
     \begin{itemize}
        \item $\rFFF$ %
        \item \PC 5,030,476,776%
        \item \RT 1h 3m 42s%
      \end{itemize}
    & 
    \begin{minipage}{.5\textwidth}
      \begin{center}\vspace*{5pt}
	\begin{tikzpicture}[scale = .35]
		\input{images/r555.tex}
	\end{tikzpicture}

      \vspace*{5pt}\end{center}
    \end{minipage}
    \\ \hline

     \begin{itemize}
        \item $\cFFFzs$ %
        \item \PC 146,835,390,410%
        \item \RTC 8h 14m 34s%
      \end{itemize}
    & 
    \begin{minipage}{.5\textwidth}
      \begin{center}\vspace*{5pt}
	\begin{tikzpicture}[scale = .35]
		\input{images/a55506.tex}
	\end{tikzpicture}

      \vspace*{5pt}\end{center}
    \end{minipage}
    \\ \hline

     \begin{itemize}
        \item $\cStzF$ %
        \item \PC 2,894,951,658%
        \item \RT 19m 39s%
      \end{itemize}
    & 
    \begin{minipage}{.5\textwidth}
      \begin{center}\vspace*{5pt}
	\begin{tikzpicture}[scale = .35]
		\input{images/a7305.tex}
	\end{tikzpicture}

      \vspace*{5pt}\end{center}
    \end{minipage}
    \\ \hline

     \begin{itemize}
        \item $\cStzzF$ %
        \item \PC 5,213,220,693%
        \item \RT 35m 12s%
      \end{itemize}
    & 
    \begin{minipage}{.5\textwidth}
      \begin{center}\vspace*{5pt}
	\begin{tikzpicture}[scale = .35]
		\input{images/a73005.tex}
	\end{tikzpicture}

      \vspace*{5pt}\end{center}
    \end{minipage}
    \\ \hline

     \begin{itemize}
        \item $\cSfzF$ %
        \item \PC 1,289,759,962,338%
        \item \RTC 2d 2h 23m 10s%
      \end{itemize}
    & 
    \begin{minipage}{.5\textwidth}
      \begin{center}\vspace*{5pt}
	\begin{tikzpicture}[scale = .35]
		\input{images/a7405.tex}
	\end{tikzpicture}

      \vspace*{5pt}\end{center}
    \end{minipage}
    \\ \hline

     \begin{itemize}
        \item $\cSfzzF$ %
        \item \PC 497,187,055,368%
        \item \RTC 16h 47m 53s%
      \end{itemize}
    & 
    \begin{minipage}{.5\textwidth}
      \begin{center}\vspace*{5pt}
	\begin{tikzpicture}[scale = .35]
		\input{images/a74005.tex}
	\end{tikzpicture}

      \vspace*{5pt}\end{center}
    \end{minipage}
    \\ \hline

     \begin{itemize}
        \item $\cSFFzF$ %
        \item \PC 14,404,843,278,891%
        \item \RTC 32d 2h 12m 28s%
      \end{itemize}
    & 
    \begin{minipage}{.5\textwidth}
      \begin{center}\vspace*{5pt}
	\begin{tikzpicture}[scale = .35]
		\input{images/a75505.tex}
	\end{tikzpicture}

      \vspace*{5pt}\end{center}
    \end{minipage}
    \\ \hline

     \begin{itemize}
        \item $\cetzFFzFF$ %
        \item \PC 26,780,028,741,288%
        \item \RTC 55d 22h 42m 28s%
      \end{itemize}
    & 
    \begin{minipage}{.5\textwidth}
      \begin{center}\vspace*{5pt}
	\begin{tikzpicture}[scale = .35]
		\input{images/a83055055.tex}
	\end{tikzpicture}

      \vspace*{5pt}\end{center}
    \end{minipage}
    \\ \hline 
      \end{longtable}
      
Runtime$^\star$ denotes total CPU time on the xvib server, a machine with four AMD Opteron 8350 2.0GHz processors, for a total of 16 CPU cores. Runtime$^\dagger$ denotes total CPU time on the colibri cluster, which has 48 Intel Xeon E5-2670 2.6GHz processors, for a total of 384 CPU cores.

\end{document}